\documentclass[11pt]{amsart}
\usepackage{amsmath}
\usepackage{amsthm}
\usepackage{amsfonts}
\usepackage{a4wide}
\usepackage{amssymb}
\usepackage{enumerate}
\newtheorem{thm}{Theorem}
\newtheorem{theorem}[thm]{Theorem}
\newtheorem{lemma}[thm]{Lemma}
\newtheorem{prop}[thm]{Proposition}
\newtheorem{corol}[thm]{Corollary}
\newtheorem{corollary}[thm]{Corollary}
\newtheorem{conj}[thm]{Conjecture}

\theoremstyle{remark}

\theoremstyle{definition}

\newtheorem{example}[thm]{Example}

\theoremstyle{remark}

\newcommand{\DD}{{\mathbb D}}
\newcommand{\OO}{{\mathcal O}}
\newcommand{\MM}{{\mathcal M}}

\newcommand{\LL}{{\mathcal L}}

\newcommand{\NN}{{\mathbb N}}

\newcommand{\cS}{{\mathcal S}}
\newcommand{\RR}{{\mathbb R}}
\newcommand{\CC}{{\mathbb C}}

\DeclareMathOperator{\sign}{sign} 
 \DeclareMathOperator{\dist}{dist}

\renewcommand{\phi}{\varphi}

\begin{document}


\title{On Carath\'eodory completeness in $\CC^n$}

\author{Armen Edigarian}

\begin{abstract}
We study $c$-completeness on domains in $\CC^n$. We reprove Sibony/Selby result on completeness on the complex plane. We also give
a characterization of $c$-completeness in Reinhardt domains.
\end{abstract}

\address{Faculty of Mathematics and Computer Science, Jagiellonian University,
\L ojasiewicza 6, 30-348 Krak\'ow, Poland}
\email{{Armen.Edigarian}@uj.edu.pl}
\thanks{Research partially supported by the National Science Center of Poland (NCN) grant
UMO-2011/03/B/ST1/04758.}


\maketitle

\section{Introduction}
Let $X$ be a complex manifold. Put
\begin{equation*}
c_X(z,w)=\sup\{\rho(f(z),f(w)): f\in\OO(X,\DD)\},
\end{equation*}
where $\rho$ denotes the hyperbolic (Poincar\'e) distance in the unit disc $\DD\subset\CC$ (see e.g. \cite{J-P},
Chapter I) and $\OO(X,\DD)$ denotes the set of all holomorphic mappings $X\to\DD$. We call $c_X$
the Carath\'eodory pseudodistance (see e.g. \cite{J-P}, Chapter II).

Note that, in general, $c_X$ is not a distance (e.g., $c_{\CC^n}\equiv0$). We say that a manifold $X$
is $c$-hyperbolic if $c_X$ is a distance, i.e., for any $z,w\in X$ there exists a bounded
holomorphic function $f$ on $X$ such that $f(z)\not=f(w)$. For $c$-hyperbolic manifold it is
natural to study a subclass of manifolds for which the metric space $(X,c_X)$ is complete. We call
this type of manifolds $c$-complete. Moreover, we say that a manifold $X$ is $c$-finitely compact if for any $x_0\in X$ and for any
$r>0$ we have $\{x\in X: c_X(x_0,x)<r\}\Subset X$. Note that any $c$-finitely compact domain is $c$-complete. Whether the inverse
implication holds for domains in $\CC^n$ is still an open question.

The following Conjecture is open for more then 30 years

\begin{conj}
Let $D\subset\CC^n$ be a bounded pseudoconvex domain with $C^\infty$ boundary. Then $D$ is $c$-complete.
\end{conj}

In our opinion, one of the main problem is lack of good conjectures, which "approximate" to the above one.
Our aim is to give such a conjecture and verify it for some class of domains.

First recall the following result, proved by P. Pflug
\begin{thm}\label{thm:1:pflug}
Let $D\subset\CC^n$ be a $c$-hyperbolic domain. Then the following conditions are equivalent:
\begin{enumerate}
\item $D$ is $c$-finitely compact;
\item for any sequence $\{z_n\}_{n=1}^\infty\subset D$ without accumulation points (w.r.t. usual topology in $D$)
there exists an $f\in\OO(D,\DD)$ with $\sup_{n}|f(z_n)|=1$.
\end{enumerate}
\end{thm}

Following Kosi\'nski and Zwonek \cite{KZ}, for a domain $D\subset\CC^n$ we say that $\zeta\in\partial D$ is a weak peak point
if there exists $f\in\OO(D)\cap C(D\cup\{\zeta\})$ such that $|f|<1$ on $D$ and $f(\zeta)=1$.

The main result of the paper is to show
\begin{thm}\label{thm:1}
Let $D\subset\CC^n$ be a $c$-hyperbolic pseudoconvex domain. Assume that
\begin{itemize}
\item $n=1$, or
\item $D$ is a Reinhardt domain.
\end{itemize}
Then the following conditions are equivalent:
\begin{enumerate}
\item $D$ is $c$-complete.
\item $D$ is $c$-finitely compact.
\item any $\zeta\in\partial D$ is a weak peak point.
\end{enumerate}
\end{thm}

Theorem~\ref{thm:1} in case $n=1$ is the result proved by N.~Sibony \cite{Sib} and M.A.~Selby \cite{Sel}.
However, their proof heavily depends on Melnikov's theorem (see e.g. \cite{GG}) which is used to show the existence of a specific measure and, in our opinion, are strictly one-dimensional. We show that the measure one can get from a generalization of Edwards' theorem.

Note that in Theorem~\ref{thm:1} the equivalence $(1)\Longleftrightarrow (2)$ is well-known. So, our result states that
these conditions are equivalent to condition $(3)$.

The above result suggests that c-finitely compactness is equivalent to the existence of weak peak functions.

\section{Non-compact version of Edwards' Theorem}
This part of the paper is motivated by \cite{GPP} and we use methods of that paper.

Let $X$ be a topological space and let $C(X)$ (resp. $C_{\CC}(X)$) be the set of all real-valued (resp. complex-valued) continuous
functions on $X$. We say that
$S\subset C(X)$ is a convex cone if $S$ contains constant functions and
$\alpha f+\beta g\in S$ for any $f,g\in S$ and any $\alpha,\beta\ge0$. Fix $x\in X$. For any $\phi\in C(X)$
we consider its envelope defined as
$$
\Phi_x(\phi)=\sup\{\psi(x): \psi\in S, \psi\le\phi\}.
$$
Note that $\Phi_x:C(X)\to[-\infty,+\infty)$ is a positive superlinear operator, i.e.,
\begin{enumerate}
\item $\Phi_x(c\phi)=c\Phi_x(\phi)$ for any $\phi\in C(X)$ and any $c\ge0$;
\item $\Phi_x(\phi_1+\phi_2)\ge \Phi_x(\phi_1)+\Phi_x(\phi_2)$ for any $\phi_1,\phi_2\in C(X)$;
\item $\Phi_x(\phi)\ge0$ for any $\phi\in C(X)$ such that $\phi\ge0$.
\end{enumerate}

We have the following result (which follows from a version of Hahn-Banach theorem, see e.g., \cite{GPP}).

\begin{theorem} Let $X$ be a topological space and let $x\in X$ be a fixed point. Assume that $S\subset C(X)$ is a convex cone.
Then
$$
\Phi_x(\phi)=\min\{L(\phi): L\text{ is a positive linear operator on }C(X), L\ge\Phi_x\}.
$$
\end{theorem}

There is an extensive literature on the study of positive linear functionals on $C(X)$ (see e.g., \cite{Hewitt}, \cite{BP}, and \cite{Res})\footnote{The author is thankful
to Piotr Niemiec and Jan Stochel for the references on linear functionals.}. However, we are interested on very special
ones, related to compact spaces. Let $X$ be a normal topological space.
We say that $X$ is of $GPP$-type if for any positive linear functional $\Lambda:C(X)\to\RR$ there exists
a compact set $K\subset X$ such that $\Lambda(\phi)=0$ for any $\phi\in C(X)$ with $\phi=0$ on $K$.
According to \cite{GPP} any locally compact $\sigma$-compact Hausdorff space is of $GPP$-type. We have
another example.

\begin{theorem} Let $D\subset\CC^n$ be a domain and let $K\subset\partial D$ be a compact set. Then $X=D\cup K$ is
of $GPP$-type.
\end{theorem}

\begin{proof}
Take sequences $R_m,r_m$ such that $R_m>r_m>R_{m+1}$ and $R_m\to 0$ (e.g.,
$R_m=\frac{1}{3^m}$ and $r_m=\frac{2}{3^{m+1}}$). Consider functions $\chi_m\in C^\infty(\RR)$ such that
$0\le \chi_m\le 1$ having the following properties:
$$
\chi_1(t)=
\begin{cases}
1& t\ge R_1\\
0& t\le r_1
\end{cases}
$$
and for any $m\ge2$
$$
\chi_m(t)=
\begin{cases}
1-\sum_{j=1}^{m-1}\chi_j(t) & t\ge R_m\\
0& t\le r_m
\end{cases}.
$$
Note that $\sum_{m=1}^\infty \chi_m(t)=1$ for $t>0$. Moreover, $\chi_m(t)=0$ for $t\ge R_{m-1}$ and $t\le r_{m}$.

Put $A_1=\{x\in X: \dist(x,K)\ge r_1\}$ and
$A_m=\{x\in X: r_m\le \dist(x,K)\le R_{m-1}\}$, $m\ge2$. Note that $A_m$, $m\ge2$, are (relatively) closed sets in $D$ and
that $\chi_m(\dist(x,K))=0$ for $x\in\CC^n\setminus A_m$. Moreover, $\cup_{m=1}^\infty A_m=D$.
We want to show that for any $m\ge1$ there exists a compact set $K_m\subset A_m$ with the following property: $\Lambda(\phi)=0$
whenever $\phi\in C(X)$, $\phi\ge0$, and $\phi=0$ on $(X\setminus A_m)\cup K_m$.

Fix $m\ge1$. For any integer $j\ge1$ we put
$$
K_{mj}=\{x\in A_m: \dist(x,\partial D)\ge\frac{1}{j}, \|x\|\le j\}.
$$
Note that $K_{mj}$ are compact sets. We also have $\cup_{j=1}^\infty K_{mj}=A_m$ for any $m\ge1$.
Let us show that there exists a $j=j(m)$ such that $\Lambda(\phi)=0$ whenever $\phi\in C(X)$, $\phi\ge0$,
$\phi=0$ on $(D\setminus A_{m})\cup K_{mj}$. Indeed, assume that for any
$j\ge1$ there exists $\phi_{j}\in C(X)$ such that $\phi_j\ge0$ on $X$, $\phi_j=0$ on
$(X\setminus A_{m})\cup K_{mj}$, and $\Lambda(\phi_j)\not=0$.
Without loss of generality we may assume that $\Lambda(\phi_j)=1$.
Consider $\phi=\sum_{j=1}^\infty \phi_j$. Note that $\phi\in C(X)$. Moreover,
$\Lambda(\phi)\ge \sum_{j=1}^N \Lambda(\phi_j)=N$ for any $N\in\NN$. Hence, $\Lambda(\phi)=+\infty$. A contradiction.

Consider a set
$$
L=\cup_{j=1}^\infty K_{mj(m)}\cup K.
$$
Note that $L$ is a compact set. It suffices to show that $\Lambda(\phi)=0$ for any $\phi\in C(X)$ such that
$\phi=0$ on $L$. First let us shot the mentioned property for $\phi\ge0$.
So, fix a function $\phi\in C(X)$ such that $\phi\ge0$ and $\phi=0$ on $L$. Fix $\epsilon>0$. Since $\phi=0$ on $K$, there exists
a $\delta>0$ such that $\phi<\epsilon$ on $\{z\in D: \dist(z,K)\le\delta\}$. We have
$$
\phi=\sum_{m=0}^\infty \phi\psi_m=\sum_{m=0}^{m_0}\phi\psi_m+\sum_{m=m_0+1}^\infty\phi\psi_m.
$$
Put $\tilde\phi=\sum_{m=m_0+1}^\infty\phi\psi_m$. For sufficiently big $m_0$ we have $\tilde\phi=0$ on $\{z\in D: \dist(z,K)>\delta\}$.
And, therefore, $0\le\tilde\phi\le\epsilon$ on $D$. So,
$$
\Lambda(\phi)=\sum_{m=0}^{m_0}\Lambda(\phi\psi_m)+\Lambda(\tilde\phi)=\Lambda(\tilde\phi)\le\epsilon\Lambda(1).
$$
Since $\epsilon>0$ was arbitrary, we get $\Lambda(\phi)=0$.

If $\phi$ is not necessary positive, then take $\phi=\phi_+-\phi_-$, where $\phi_+=\max\{\phi,0\}$ and $\phi_-=\phi_+-\phi$.
Then $L(\phi)=L(\phi_+)-L(\phi_-)=0$.
\end{proof}

We need also the following version of the Riesz representation theorem.
\begin{prop} Let $X$ be a normal topological space and let $L:C(X)\to\RR$ be a positive linear functional. Assume that
there exists a compact set $K\subset X$ such that $L(\phi)=0$ whenever $\phi\in C(X)$, $\phi=0$ on $K$. Then there exists a Borel
finite measure $\mu$ with support in $K$ such that
$$
L(\phi)=\int\phi (x)d\mu(x)\quad\text{ for any }\phi\in C(X).
$$
\end{prop}

\begin{proof} Let us define a positive linear operator $\widetilde L:C(K)\to\RR$. Fix $\phi\in C(K)$. From the normality
of $X$ there exists $\widetilde\phi\in C(X)$ such that $\widetilde\phi=\phi$ on $K$. We put $\widetilde L(\phi)=L(\widetilde\phi)$.
Now we use the classical Riesz representation theorem for the operator $\widetilde L$ and get $\mu$.
\end{proof}

For each convex cone $S$ and a point $x\in X$ we associate two sets:
\begin{enumerate}
\item $J_x^{S}(X)$ - the set of all \emph{Jensen measures} with barycenter at $x$ which consists of all Borel probability measures $\mu$ with compact support such that $\psi(x)\le\int \psi d\mu$ for any $\psi\in X$;
\item $R_x^{S}(X)$ - the set of all \emph{representing measures} with barycenter at $x$ which consists of all Borel probability measures $\mu$ with compact support such that $\psi(x)=\int \psi d\mu$ for any $\psi\in X$;
\end{enumerate}

Note that $R_x^{S}(X)\subset J_x^{S}(X)$.

In 1965 Edwards \cite{E} proved the following result:
\begin{theorem}
Let $X$ be a compact topological space and let $S\subset C(X)$ be a convex cone. Assume that $\phi$ is a lower semicontinuous function on $X$. Then
$$
\Phi_x(\phi)=\min\Big\{\int \phi d\mu: \mu\in J_{x}^S(X)\Big\}.
$$
\end{theorem}

In 2013 Gogus, Perkins, and Poletsky \cite{GPP} proved the following non-compact version of Edwards' theorem
\begin{theorem}\label{thm:GPP}
Let $X$ be a locally compact $\sigma$-compact Hausdorff space and let  $S\subset C(X)$ be a convex cone.
Assume $\phi$ be a continuous function on $X$. Then
$\Phi_x(\phi)\equiv-\infty$ or
$$
\Phi_x(\phi)=\min\Big\{\int \phi d\mu: \mu\in J_{x}^S(X)\Big\}.
$$
\end{theorem}

From the above results we get
\begin{theorem} Let $X$ be a normal topological space of $GPP$-type and let $S\subset C(X)$ be a convex cone.
If $\phi\in C(X)$  then for any $x\in X$ we have $\Phi_x(\phi)=-\infty$ or
$$
\Phi_x(\phi)=\inf\Big\{\int\phi d\mu: \mu\in J_{x}^{\cS}\Big\}\quad\text{ for any }x\in X.
$$
\end{theorem}

We have the following important corollary, which is a non-compact counterpart of Theorems II.11.1 and II.11.3 in \cite{Gam}.
\begin{corollary}\label{cor:10}
Let $D\subset\CC^n$ be a domain and let $\zeta\in\partial D$. We put $X=D\cup\{\zeta\}$ and
$S=\{\Re(f): f\in A(X)\}$, where $A(X)=H^\infty(D)\cap C_{\CC}(X)$.
Then the following conditions are equivalent:
\begin{enumerate}
\item $R_{\zeta}^S(X)=\{\delta_\zeta\}$;
\item there exists a function $f\in A(X)$ such that $f(\zeta)=1$ and $|f|<1$ on $D$, i.e., $\zeta$ is a peak point for $A(X)$;
\item there exist a $c\in(0,1)$ and an $M\ge1$ such that for any $r>0$ there exists an $f\in A(X)$ with
\begin{enumerate}
\item $\|f\|_{X}\le M$;
\item $f(\zeta)=1$;
\item $|f|\le c$ on $X\setminus\DD(\zeta,r)$.
\end{enumerate}
\end{enumerate}
\end{corollary}

\section{Carath\'eodory completeness on the complex plane}

Before we state main results of this section, let us recall some notions and results from one-dimensional analysis.

Let $D\subset\CC$ be a domain and let $\zeta\in\partial D$. Recall that if
$$
\limsup_{r\to0}\frac{\LL(\DD(\zeta,r)\setminus D)}{r^2}>0
$$
then $\zeta$ is a peak point for $A(\overline{D})=\OO(D)\cap C(\overline{D})$. Here, $\LL$ denotes the Lebesgue
measure on the complex plane. Essentially, its Curtis's Criterion and the relation between analytic capacity and the Lebesgue
measure (see e.g. \cite{Gam}, Corollary VIII.4.2).

So, if $\zeta$ is not a peak point for $A(\overline{D})$ then
\begin{equation}\label{eq:5}
\lim_{r\to0}\frac{\LL(\DD(\zeta,r)\cap D)}{\pi r^2}=1.
\end{equation}

Let $\MM$ denotes the set of all positive probability measure in $\CC$ with compact support and let
$\mu\in\MM$.
We define its Newton potential as $M(\xi)=\int\frac{1}{|w-\xi|}d\mu(w)$. The following result is a corollary of Fubini's theorem
(see e.g. \cite{Sto}, Lemma 26.16).
\begin{prop}\label{prop:11} For any $\zeta\in\CC$ we have
$$
\lim_{r\to0}\frac{1}{\pi r^2}\int_{\DD(\zeta,r)}|w-\zeta|\cdot M(w)d\LL(w)=\mu(\{\zeta\}).
$$
In particular, if $\mu(\{\zeta\})=0$ then for any $\epsilon>0$ the set
$$
\Pi(\epsilon)=\{w\in\CC: |w-\zeta|\cdot M(w)>\epsilon\}
$$
has the property
$$
\lim_{r\to0}\frac{\LL(\Pi(\epsilon)\cap\DD(\zeta,r))}{\pi r^2}=0.
$$
\end{prop}
As a Corollary of Proposition~\ref{prop:11} and \eqref{eq:5} we get the following
\begin{corol}\label{cor:12} Let $D\subset\CC$ be a domain, let $\zeta\in\partial D$, and let $\mu$ be a Borel measure with compact
support in $D\cup\{\zeta\}$ such that $\mu(\{\zeta\})=0$. Assume that $\zeta$ is not a peak point for $A(\overline D)$. Then
$$
\lim_{r\to0}\frac{\LL(D\cap\DD(\zeta,r)\setminus \Pi(\epsilon))}{\pi r^2}=1.
$$
\end{corol}

The main result of this section is the following.

\begin{thm}\label{thm:2} Let $D\subset\CC^n$ be a domain. Consider the following conditions:
\begin{enumerate}
\item\label{3} for any $\zeta\in\partial D$ there exist no a Borel probability measure $\mu$ with compact support in $D\cup\{\zeta\}$ such that $\mu\not=\delta_{\zeta}$ and
$$
|f(\zeta)|\le\int_{D\cup\{\zeta\}} |f(w)| d\mu(w)\quad\text{ for any }f\in A(D\cup\{\zeta\}).
$$
\item\label{3a} for any $\zeta\in\partial D$  there exist no a Borel probability measure $\mu$ with compact support in $D\cup\{\zeta\}$
such that $\mu\not=\delta_{\zeta}$ and
$$
f(\zeta)=\int_{D\cup\{\zeta\}} f(w)d\mu(w)\quad\text{ for any }f\in A(D\cup\{\zeta\}).
$$
\item\label{4} there exists an $f\in A(D\cup\{\zeta\})$ such that $f(\zeta)=1$ and $|f|<1$ on $D$.
\item\label{2} $D$ is $c$-finitely compact.
\item\label{1} $D$ is $c$-complete.
\end{enumerate}
Then $\eqref{3}\implies\eqref{3a}\implies\eqref{4}\implies\eqref{2}\implies\eqref{1}$. Moreover,
if $n=1$ then $\eqref{1}\implies\eqref{3}$ and, therefore, all the above conditions are equivalent.
\end{thm}

\begin{proof}[Proof of Theorem \ref{thm:2}] Note that the implications $\eqref{3}\implies\eqref{3a}$ and
$\eqref{4}\implies\eqref{2}\implies\eqref{1}$ are immediate. The implication $\eqref{3a}\implies\eqref{4}$ follows from
Corollary~\ref{cor:10}.

Assume that $n=1$. Let us prove $\eqref{1}\implies\eqref{3}$. Assume that there exists a positive probability measure $\mu$ such that $\mu(D)=1$ and
$$
|f(\zeta)|\le\int_D |f|d\mu\quad\text{ for any }f\in A(D\cup\{\zeta\}).
$$
Fix $f\in A(D\cup\{\zeta\})$.  Then there exists a sequence $f_n\in H^\infty(D)$ such that $\|f_n\|_{D}\le 17\|f\|_{D}$, $f_n$ extends to be analytic in a neighborhood of $\zeta$ and $f_n$ converges uniformly to $f$ on any set of type $D\setminus\DD(\zeta,\epsilon)$, where $\epsilon>0$.

For any $\eta\in D$ we put
$$
g_n(z)=\frac{f_n(z)-f_n(\eta)}{z-\eta}. 
$$
Note that $g_n\in H^\infty(D\cup\{\zeta\})$.
Then
$$
|g_n(\zeta)|\le\int_{D}|g_n(w)|d\mu(w)\le 2\|f_n\|_{\infty} M(\eta)\le 34\|f\|_{\infty} M(\eta)
$$
and, therefore,
$$
|f_n(\zeta)-f_n(\eta)|\le |\zeta-\eta|\cdot 2\|f_n\|_{\infty} M(\eta)\le |\zeta-\eta|\cdot 34\|f\|_{\infty} M(\eta).
$$
For any $\eta_1,\eta_2\in D$ we have
$$
|f(\eta_1)-f(\eta_2)|\le 34\|f\|_{\infty}\cdot\big(|\zeta-\eta_1|\cdot  M(\eta_1)+|\zeta-\eta_2|\cdot M(\eta_2)\big).
$$
According to Corollary~\ref{cor:12} we can take a sequence $\{\eta_\nu\}$ such that $\eta_\nu\to\zeta$ and  $|\zeta-\eta_\nu|\cdot  M(\eta_\nu)\le\frac{1}{2^\nu}$. Then $\{\eta_\nu\}$ is a c-Cauchy sequence. A contradiction.
\end{proof}

\section{The Reinhardt domains}
A domain $D\subset\CC^n$ is called \emph{Reinhardt} if $(\lambda_1 z_1,\dots,\lambda_n z_n)\in D$ for all points
$z=(z_1,\dots,z_n)\in D$ and any $|\lambda_1|=\dots=|\lambda_n|=1$.
Let us denote
$$
V_j=\{z\in\CC^n: z_j=0\}, \quad j=1,\dots,n.
$$

Recall the following result of W.~Zwonek (see \cite{Zwo1}, Theorem~2).

\begin{theorem}\label{thm:14} Let $D$ be a hyperbolic pseudoconvex Reinhardt domain. Then the
following conditions are equivalent:
\begin{enumerate}
\item $D$ is c-finitely compact;
\item $D$ is c-complete;
\item $D$ is bounded and for any $j\in\{1,\dots,n\}$
\begin{equation}\label{eq:3:1}
\text{if }\overline{D}\cap V_j\not=\varnothing,\text{ then } D\cap V_j\not=\varnothing.
\end{equation}
\end{enumerate}
\end{theorem}

All we have to prove is that condition $(3)$ in Theorem~\ref{thm:14} implies that any $\zeta\in\partial D$ is a weak peak point.
Without loss of generality we may assume that
\begin{enumerate}[(i)]
\item $D\subset\DD^n$;
\item $D\cap V_j\not=\varnothing$, $j=1,\dots,m$.
\item $\overline{D}\cap V_j=\varnothing$, $j=m+1,\dots,n$.
\end{enumerate}

Fix a boundary point $\zeta=(\zeta_1,\dots,\zeta_n)\in\partial D$.
First assume that $\zeta_1\cdot\dots\cdot\zeta_n\not=0$. Then there exist
$\xi_1,\dots,\xi_n\in\RR$ such that $\phi(z)<\phi(\zeta)$ for any $z\in D$, where
$\phi(z)=|z_1|^{\xi_1}\cdot\dots\cdot|z_n|^{\xi_n}$.
Note that $\xi_1,\dots,\xi_m\ge0$ and that there exists a $\delta_0>0$ such that for any
$z\in D$ we have $|z_j|\ge \delta_0$ for $j=m+1,\dots,n$. Without loss of generality we may
assume that $\xi_1,\dots,\xi_\ell>0$ and $\xi_{\ell+1}=\dots=\xi_m=0$.

Fix a point $\eta\in D$ such that $\eta_1\cdot\ldots\cdot \eta_n\not=0$.
Put $R=\frac{\phi(\eta)}{\phi(\zeta)}<1$. We want to show that
there exists a sequence $g_k\in\OO(D)\cap C(D\cup\{\zeta\})$ such that
\begin{enumerate}
\item $\|g_k\|_{D}\to1$ when $k\to\infty$;
\item $g_k(\zeta)=1$ for any $k\ge1$;
\item $|g_k(\eta)|\le R$ for any $k\ge1$.
\end{enumerate}

Fix $\epsilon>0$. There exist $\beta_{1},\dots,\beta_{n}\in\RR$ and $q\in\NN$, $q\ge2$,
such that $\sign\beta_j=\sign \xi_j$ and
$$
|q\xi_{j}-\beta_{j}|\le \epsilon\quad\text{ for any } j=1,\dots,n.
$$
Put $f(z)=z_1^{\beta_{1}}\cdot\dots\cdot z_n^{\beta_{n}}$ and $g(z)=\frac{f(z)}{f(\zeta)}$.
We have
$$
|f(z)|=\phi(z)^{q}|z_1|^{\beta_{1}-\xi_1 q}\cdot\dots\cdot |z_n|^{\beta_{n}-\xi_n q}.
$$
Hence,
$$
|g(z)|=\left(\frac{\phi(z)}{\phi(\zeta)}\right)^{q}\cdot
\left(\frac{|z_1|}{|\zeta_1|}\right)^{\beta_{1}-\xi_1 q}
\cdot\ldots\cdot
\left(\frac{|z_n|}{|\zeta_n|}\right)^{\beta_{n}-\xi_n q}
$$
and, in particular,
$$
|g(z)|\le\left(\frac{|z_1|}{|\zeta_1|}\right)^{\beta_{1}-\xi_1 q}
\cdot\ldots\cdot
\left(\frac{|z_n|}{|\zeta_n|}\right)^{\beta_{n}-\xi_n q}.
$$

$1^o$ Let us show that $|g(\eta)|<R$. Indeed,
$$
|g(\eta)|\le R^q\cdot
\left(\prod_{j=1}^n \max\Big\{\frac{|\eta_j|}{|\zeta_j|}, \frac{|\zeta_j|}{|\eta_j|}\Big\}\right)^{\epsilon}.
$$
For sufficiently small $\epsilon>0$ we have $|g(\eta)|<R$.

$2^o$ Let us now estimate $\|g\|_{D}$. There exists $\delta_1>$ such that for any $z\in G$ with
$\min\{|z_1|,\dots,|z_\ell|\}\le\delta_1$ we have
$$
|z_1|^{\beta_1}\cdot\ldots\cdot |z_n|^{\beta_n}\le\phi(\zeta)^{q}.
$$
Note that
$$
|z_1|^{\beta_1}\cdot\ldots\cdot |z_n|^{\beta_n}\le
|z_1|^{\beta_1}\cdot\ldots\cdot |z_\ell|^{\beta_\ell}\cdot \delta_0^{\sum_{j=m+1}^n \beta_j}\le
\delta_1^{q\min\{\zeta_1,\dots,\zeta_\ell\}-\epsilon}\cdot \delta_0^{q\sum_{j=m+1}^n (\zeta_j-\frac{\epsilon}{q})}.
$$
For sufficiently small $\delta_1>0$ we have
$$
\delta_1^{\min\{\zeta_1,\dots,\zeta_\ell\}-\frac{\epsilon}{2}}\cdot \delta_0^{\sum_{j=m+1}^n (\zeta_j-\frac{\epsilon}{2})}
\le\phi(\zeta).
$$
Fix such a $\delta_1>0$. Take $z\in D$ such that $\min\{|z_1|,\dots,|z_\ell|\}\le\delta_1$. Then $|f(z)|\le\phi(\zeta)^q$.
Therefore,
$$
|g(z)|\le \prod_{j=1}^n |\zeta_j|^{q\zeta_j-\beta_j}\le
\left(\frac{1}{\prod_{j=1}^n |\zeta_j|}\right)^\epsilon.
$$
Take $z\in D$ such that $|z_1|,\dots,|z_\ell|\ge\delta_1$. Then
$$
|g(z)|\le \left(\prod_{j=1}^n \max\Big\{\frac{\delta_1}{|\zeta_j|}, \frac{|\zeta_j|}{\delta_1}\Big\}\right)^{\epsilon}.
$$
All in all, we have $\|g\|_{D}\to1$ when $\epsilon\to0$.

In case, $\zeta_1\cdot\ldots\zeta_n=0$, consider a projection $\pi:\CC^n\to\CC^{n-1}$. Take $\widetilde D=\pi(D)$.

For the construction of a peak function we need the following simple, however, useful result.

\begin{theorem}
Let $\Omega\subset\CC^n$ be a domain and let $\zeta\in\partial\Omega$, $\eta\in\Omega$, $r\in[0,1)$ be fixed. Assume that
for any $\epsilon>0$ there exists $f_\epsilon\in\OO(\Omega)\cap C(\Omega\cup\{\zeta\})$ such that
\begin{enumerate}
\item $|f_\epsilon|<1+\epsilon$ on $\Omega$;
\item $f_\epsilon(\zeta)=1$;
\item $|f_\epsilon(\eta)|<r$.
\end{enumerate}
Then there exists $F\in\OO(\Omega)\cap C(\Omega\cup\{\zeta\})$ such that
\begin{enumerate}
\item $|F|<1$ on $\Omega$;
\item $F(\zeta)=1$.
\end{enumerate}
\end{theorem}

It is easy to check
\begin{lemma} Let $\epsilon>0$. Then
$$
\Re\left(\frac{1+z}{1-z}\right)>\frac{1}{\epsilon}\Longleftrightarrow |z-\frac{1}{1+\epsilon}|<\frac{\epsilon}{1+\epsilon}.
$$
\end{lemma}

\begin{proof}[Proof of Theorem]
Without loss of generality we may assume that $r=0$. Indeed, take a sequence
$$
g_\epsilon(\lambda)=a(1+\epsilon)\cdot\frac{f_\epsilon(\lambda)-f_\epsilon(\eta)}
{(1+\epsilon)^2-\overline{f_\epsilon(\eta)}f_\epsilon(\lambda)},
$$
where
$$
a=\frac{1}{1+\epsilon}\cdot\frac{(1+\epsilon)^2-\overline{f_\epsilon(\eta)}}{1-f_\epsilon(\eta)}.
$$
Note that $g_\epsilon(\eta)=0$, $g_\epsilon(\zeta)=1$, and
$$
|g_\epsilon(\lambda)|\le |a|=
\left|\frac{1}{1+\epsilon}\cdot\frac{(1+\epsilon)^2-f_\epsilon(\eta)}{1-f_\epsilon(\eta)}\right|\le
\frac{1}{1+\epsilon}\left(1+\frac{2\epsilon+\epsilon^2}{1-r}\right)\to1\quad\text{when }\epsilon\to0.
$$

So, we assume that $r=0$. Put $\epsilon_k=\frac{1}{4^k}$. Then there exists
$f_k=f_{\epsilon_k}\in\OO(\Omega)\cap C(\Omega\cup\{\zeta\})$ such that
\begin{enumerate}
\item $|f_k|<1+\epsilon_k$ on $\Omega$;
\item $f_k(\zeta)=1$;
\item $f_k(\eta)=0$.
\end{enumerate}
Put $U_k=\{z\in\Omega\cup\{\zeta\}: |f_k(z)-1|<\epsilon_k\}$ and
$$
h(z)=\sum_{k=1}^\infty \frac{1}{2^k}\cdot \frac{1+\tilde f_k(z)}{1-\tilde f_k(z)},
$$
where $\tilde f_k=\frac{f_k}{1+\epsilon_k}$. Note that
$\Re h(z)\ge \frac{1}{2^k}\cdot \Re\left(\frac{1+\tilde f_k(z)}{1-\tilde f_k(z)}\right)\ge 2^k$ for $z\in U_k$.
Put $F=\frac{h-1}{h+1}$. Then $|F-1|\le\frac{2}{|h+1|}\le\frac{2}{|\Re h+1|}\le \frac{2}{2^k+1}$ on $U_k$.
\end{proof}

Using Sibony's ideas we show that there exists a domain $D\subset\CC^2$ and a boundary point such that a weak peak
function exists, however, peak function does not exist.

\begin{example} Fix an irrational number $\alpha>0$.
Let $D\subset\CC^2$ be a domain and let $(z_0,w_0)\in\partial D$. Assume that there exists
a neighborhood $U\subset\CC^2$ of $(z_0,w_0)$ such that
$$
D\cap U=\{(z,w)\in\CC^2: |z|\cdot |w|^\alpha<1\}\cap U.
$$
We want to show that there does not exist a holomorphic function $f\in\OO(D)\cap C(D\cup U)$ such that $|f|<1$ on $D$
and $f(z_0,w_0)=1$.

Indeed, assume that such a function exists. There exists a neighborhood $V\subset\CC$ of the origin such that
$(z_0e^{-\alpha\lambda},w_0e^{\lambda})\in U$ whenever $\lambda\in V$. For sufficiently big $n\in\NN$ consider functions
$$
\psi_n(\lambda)= f(z_0e^{-\alpha\lambda},(1-\frac1n)w_0e^{\lambda}).
$$
Note that $\psi_n:V\to D\cap U$ is a holomorphic mapping. Hence, there exists a subsequence $\{n_k\}$ such that
$\psi_{n_k}$ tends locally uniformly on $V$ to a holomorphic mapping $\psi:V\to\CC^2$. It is easy to see (use continuity of $f$) that
$$
\psi(\lambda)=f(z_0e^{-\alpha\lambda},w_0e^{\lambda}).
$$
So, we get that $\psi:V\to \overline{D\cap U}$ is a holomorphic mapping such that $|\psi|\le1$ and $\psi(0)=1$.
Hence, $\psi\equiv1$. Since $\alpha$ is irrational, we get $\{(z_0e^{-\alpha\lambda},w_0e^{\lambda}): \lambda\in V\}$
is dense in a neighborhood of $(z_0,w_0)$. From the continuity of $f$ we get that $f=1$ on a relatively open subset $\partial D$ containing
$(z_0,w_0)$. Then a function $f(z_0,\lambda w_0)=1$ on the open subset of the unit circle containing $1$. Hence, $f(z_0,\lambda w_0)=1$ everywhere.
A contradiction.
\end{example}

\bibliographystyle{amsplain}

\begin{thebibliography}{10}

\bibitem{BP} P.~Berti, P.~Rigo, \textit{Integral representation of linear functionals on spaces of unbounded functions},
PAMS 128 (2000), 3251--3258.

\bibitem{Bis} E.~Bishop, \textit{A minimal boundary for function algebras}, Pacific J. Math. 9 (1959), 629--642.

\bibitem{E} Edwards D. A., \textit{Choquet boundary theory for
    certain spaces of lower semicontinuous functions, in
    Function Algebras,} Proc. Internat. Sympos. on Function
    Algebras, Chicago, (1966), 300–309.

\bibitem{Gam} T.W.~Gamelin, \textit{Uniform algebras}, Prentice-Hall, 1969.

\bibitem{GG} T.W.~Gamelin, J.~Garnett, \textit{Distinguished homomorphisms and fiber algebras}, Trans. Amer. Math. Soc (1970), 455--474.

\bibitem{GPP} N.~Gogus, T.~Perkins, E.~Poletsky, \textit{Non-compact versions of Edwards' Theorem}, Positivity 17 (2013), 459--473.

\bibitem{J-P} M.~Jarnicki \& P. Pflug, \textit{Invariant distances and metrics in complex analysis},
De Gruyter Expositions in Mathematics 9, 2nd ext. edition, 2013.

\bibitem{Hewitt} E.~Hewitt, \textit{Linear functionals on spaces of continuous functions}, Fund. Math. 37 (1950), 161--189.

\bibitem{KZ} \L.~Kosi\'nski, W.~Zwonek, \textit{Proper holomorphic mappings vs. peak points and Shilov boundary},
Ann. Polon. Math. 107 (2013), 97--108.

\bibitem{P} E. A. Poletsky, \textit {Plurisubharmonic functions
    as solutions of variational problems,} Proc. Sympos. Pure
    Math. {\bf 52} (1991), Part 1, 163--171.

\bibitem{W} F. Wikstr\"om, \textit{Jensen measures and boundary values of plurisubharmonic functions,}
    Ark. Mat. {\bf 39}(2001), 181--200.

\bibitem{Res} P.~Ressel, \textit{Integral representations on convex semigroups}, Math. Scand. 61 (1987), 93--111.

\bibitem{Rud} W.~Rudin, \textit{Real and Complex Analysis}, McGraw-Hill, 1987.

\bibitem{Sel} M.A. Selby, \textit{On completeness with respect to the Carath\'eodory metric}, Canad. Math.
Bull. 17 (1974), 261--263.

\bibitem{Sib} N. Sibony, \textit{Prolongement de fonctions holomorphes born\'ees et metrique de Carath\'eodory},
Invent. Math. 29 (1975), 205--230.

\bibitem{Sto} E.L.~Stout, \textit{The theory of uniform algebras}, Bogden \& Quigley, 1971.

\bibitem{Zwo1} W. Zwonek, \textit{On Caratheodory completeness of pseudoconvex Reinhardt domains}, PAMS 128 (2000), 857--864.

\end{thebibliography}

\end{document}